\documentclass[10pt,a4paper]{article}
\usepackage[latin1]{inputenc}
\usepackage{fullpage}
\usepackage{mathrsfs}
\usepackage{amsmath}
\usepackage{amsfonts}
\usepackage{amssymb}
\usepackage{amsthm}
\usepackage{graphicx}
\usepackage{epstopdf}
\usepackage[affil-it]{authblk}
\graphicspath{C:/Users/evanc/OneDrive/Documents/Vanderbilt HA Conference}
\DeclareGraphicsExtensions{.eps,.pdf,.png,.jpg}
\author{Evan Camrud%
	\thanks{Electronic address: \texttt{ecamrud@iastate.edu}; Corresponding author} }
\affil{Mathematics Department, Iowa State University, Ames, IA 50011}
\title{Applications of a distributional fractional derivative to Fourier analysis and its related differential equations}

\newtheorem{Theorem}{Theorem}

\theoremstyle{definition}
\newtheorem{Definition}{Definition}
\newtheorem{Remark}{Remark}

\begin{document}

\maketitle

\textbf{Abstract.} A new definition of a fractional derivative has recently been developed, making use of a fractional Dirac delta function as its integral kernel. This derivative allows for the definition of a distributional fractional derivative, and as such paves a way for application to many other areas of analysis involving distributions. This includes (but is not limited to): the fractional Fourier series (i.e. an orthonormal basis for fractional derivatives), the fractional derivative of Fourier transforms, and fundamental solutions to differential equations such as the wave equation. This paper observes new results in each of these areas.

\section{Introduction}

Recent research in fractional differentiation has produced an operator whose results intersect with past fractional derivatives, and also allows for an extension to fractional derivatives \textit{in the sense of distributions} and \textit{fractional derivatives of distributions} \cite{Camrud}. The operator is defined formally as:

\begin{Definition}
	
	Let $f\in X_t(\Omega)\subseteq X^{-c}(\Omega)$ a distributional function space with trivial constant, and let $z_0\in\partial \Omega$ (note $z_0$ may be infinite).
	
	For $\alpha\in\mathbb{C}$, the $\alpha^{th}$ distributional differintegral of $f(z)$, with respect to the variable $z$, is
	
	\begin{equation}
	\begin{split}
	\text{\LARGE\S}_z ^\alpha f(z)&=\frac{1}{\Gamma(\alpha)}\int_{z_0}^z f(\zeta)(z-\zeta)^{\alpha-1}d\zeta\\	
	&=\frac{1}{\Gamma(\alpha)}\oint_{\gamma} f(\zeta)(z-\zeta)^{\alpha-1}H(z-\zeta)d\zeta
	\end{split}
	\end{equation}
	where $\gamma$ is a simple closed curve in $\Omega$ containing the points $z_0$ and $z$, with $H(z-\zeta)$ becoming the real-valued Heaviside step function when $\gamma$ is parameterized by a real variable.\footnote{That is, if $\gamma(t):[t_0,t_1]\to \Omega$ with $\gamma(t_0)=\gamma(t_1)=z_0$ and $\gamma(t')=z$ for $t'\in[t_0,t_1]$, then $H(z-\zeta)=H\big(z-\gamma(t)\big)=1$ if $t\in[t_0,t')$ and $H(z-\zeta)=H\big(z-\gamma(t)\big)=0$ if $t\in(t',t_1]$.}
\end{Definition}

However, in an applied setting, it is much easier to view the definition as

\begin{Definition}
	For a test function $\phi\in\mathcal{D}(\mathbb{R})$, the distributional differintegral is
	\begin{equation}
	\phi^{(-\alpha)}(x)=\int_{\mathbb{R}}\phi(t)\delta^{(-\alpha)}(x-t)dt
	\end{equation}
	where $\delta^{(-\alpha)}(x)=\frac{x^{\alpha-1}}{\Gamma(\alpha)}H(x)$, with $H(x)$ the Heaviside step function.
\end{Definition}
For functions in a space that are not test functions, one defines $f^{(-\alpha)}(x)=\lim_{n\to\infty}\phi_n^{(-\alpha)}(x)$ where $\{\phi_n\}$ is a sequence of test functions converging to $f$ in norm. Similarly, one defines the fractional differintegral of distributions via $T^{(-\alpha)}[\phi]=e^{-i\pi \alpha}\cdot T[\phi^{(-\alpha)}]$ for all test functions $\phi$.

To ensure that the operator satisfies the index law, that is $\text{\Large\S}_x^\alpha\text{\Large\S}_x^\beta=\text{\Large\S}_x^{\alpha+\beta}$, the derivative direction of the operator needs to be injective. This causes a problem when dealing with constants since $\frac{d}{dx}1=\frac{d}{dx}0=0$. To sidestep this issue, integer-valued derivatives of constants are given a ``label'' that keeps them identifiable in the codomain of the operator. These are called the \textit{Zero Functions} and they occur in the form

\begin{equation}
\emptyset(x)=\frac{d}{dx}1=\frac{x^{-1}}{\Gamma(0)}
\end{equation}
and more importantly as a fractional differintegral
\begin{equation}
\emptyset^{(-\alpha)}(x)=\frac{x^{\alpha-1}}{\Gamma(\alpha)}
\end{equation}
such that $\delta^{(-\alpha)}(x)=\emptyset^{(-\alpha)}(x)H(x)$. Since the operator is linear, one may use the above to define the differintegral of an arbitrary constant.

\section{Fractional Fourier series}

It is well known that if $f\in C(\mathbb{T})$ then the Fourier series $\sum_{n=-\infty}^\infty c_n e^{i n x}$ with $c_n=\int_{\mathbb{T}}f(x)e^{-i n x}dx$ uniformly converges to $f$ on $\mathbb{T}$. Furthermore, if one restricts the functions to $L^2(\mathbb{T})$, then the sequence $\{e^{i n x}\}_{n=-\infty}^\infty$ forms an orthonormal basis for the Hilbert space.

It seems natural to ask whether fractional derivatives of the exponentials can serve as linearly dense subset for the natural derivative extension of $L^2(\mathbb{T})$, namely the Sobolev space $W^{\alpha,2}(\mathbb{T})$.

First note that from the above definition of the differintegral, we have
\begin{equation}
\text{\LARGE\S}_x^\alpha e^{i n x}=(i n)^{-\alpha} e^{i n x}=n^{-\alpha} e^{i (nx-\alpha\frac{\pi}{2})}
\end{equation}
where any necessary branch cuts are taken along the negative imaginary axis: $\arg z\in (-\frac{\pi}{2},\frac{3\pi}{2})$.

If we allow ourselves linear combinations with complex coefficients, however, it may be seen that
\begin{equation}
\sum_{n=-\infty}^\infty c_n (in)^{-\alpha} e^{inx}=\sum_{n=-\infty}^\infty d_n e^{inx}
\end{equation}
for $d_n=(in)^{-\alpha}c_n$. Therefore any linear combination of $\Big\{\text{\Large\S}_x^\alpha e^{inx}\Big\}_{n=-\infty}^\infty$ is also a linear combination of $\{e^{i n x}\}_{n=-\infty}^\infty$, which is an orthonormal subset of $W^{\alpha,2}(\mathbb{T})$.

\begin{Theorem}
	The set $\{e^{i n x}\}_{n=-\infty}^\infty$ forms an orthonormal basis for $W^{\alpha,2}(\mathbb{T})$. More importantly, the series $f^{(\beta)}(x)=\sum_{n=-\infty}^\infty d_n e^{inx}$ is uniformly convergent for all $\beta\in[0,\alpha]$ and the coefficients are of the form $d_n=(in)^\beta \int_{\mathbb{T}}f(x)e^{-inx}dx$.
\end{Theorem}
\begin{proof}
	
It is trivial to show that $\{e^{i n x}\}_{n=-\infty}^\infty$ is orthonormal, and since $W^{\alpha,2}(\mathbb{T})\subset L^2(\mathbb{T})$ we know that $\{e^{i n x}\}_{n=-\infty}^\infty$ is a basis.

Let $f\in W^{\alpha,2}(\mathbb{T})$ and $\beta\in[0,\alpha]$. Then $f^{(\beta)}\in L^2(\mathbb{T})$.

Since $\{e^{i n x}\}_{n=-\infty}^\infty$ is an orthonormal basis for $L^2(\mathbb{T})$ there exists $d_n=\int_{\mathbb{T}}f^{(\beta)}(x)e^{-inx}dx$ such that $\sum_{n=-\infty}^\infty d_n e^{inx}=f^{(\beta)}(x)$ uniformly.

Since the sum converges uniformly, by the integral definition of $\text{\Large\S}_x^\beta$ we have that
\begin{equation}
\begin{split}
f(x)&=\text{\LARGE\S}_x^\beta f^{(\beta)}(x)=\text{\LARGE\S}_x^\beta\bigg[\sum_{n=-\infty}^\infty d_n e^{inx}\bigg]=\sum_{n=-\infty}^\infty d_n\text{\LARGE\S}_x^\beta e^{inx}\\
&=\sum_{n=-\infty}^\infty d_n (in)^{-\beta}e^{inx}=\sum_{n=-\infty}^\infty c_n e^{inx}
\end{split}
\end{equation}
where $c_n=(in)^{-\beta}d_n$ or rather $d_n=(in)^\beta c_n=(in)^\beta\int_{\mathbb{T}}f(x)e^{-inx}dx$.

\end{proof}

In the sense of distributions, as $\sum_{n=-\infty}^\infty e^{inx}=\delta(x)$ we may also assert the following:

\begin{Theorem}
	\begin{equation}
	\sum_{n=-\infty}^\infty n^{\alpha}e^{i(nx+\alpha\frac{\pi}{2})}=\delta^{(\alpha)}(x)
	\end{equation}
	in the sense of distributions.
\end{Theorem}

\begin{proof}
	Since $\int_{\mathbb{R}} \phi(t)\delta^{(\alpha)}(x-t)dt=\phi^{(\alpha)}(x)$ for all test functions $\phi\in\mathcal{D}(\mathbb{T})$, and since the series $\phi(x)=\sum_{n=-\infty}^\infty c_n e^{inx}$ (for $c_n=\int_{\mathbb{T}}\phi(x)e^{-inx}dx$) is uniformly convergent, we have that
	
	\begin{equation}
	\begin{split}
	\phi^{(\alpha)}(x)&=\int_\mathbb{R} \phi(t),\delta^{(\alpha)}(x-t)dt=\int_{\mathbb{R}}\bigg[ \sum_{n=-\infty}^\infty c_n e^{int},\delta^{(\alpha)}(x-t)\bigg]dt\\
	&=\sum_{n=-\infty}^\infty c_n\cdot\bigg[\int_{\mathbb{R}} e^{int}\delta^{(\alpha)}(x-t)dt\bigg]=\sum_{n=-\infty}^\infty c_n n^\alpha e^{i(nx+\alpha\frac{\pi}{2})}
	\end{split}
	\end{equation}
	Since this property holds for all $\phi\in\mathcal{D}(\mathbb{T})$ we have convergence in the sense of distributions.
\end{proof}

\section{Differintegral of Fourier transforms}

It is often the case that Fourier transforms serve as a gateway for fractional differentiation, as a multiplication operator $T_{\alpha}[\widetilde{f}(\omega)]=(i\omega)^\alpha \widetilde{f}(\omega)$ in Fourier space is equivalent to differentiation in the domain space. One would hope that a similar result follows from the above differintegral. Indeed we see the following:

\begin{Theorem}
	\begin{equation}
	\mathcal{F}[f^{(\alpha)}(t)]=(i\omega)^\alpha \widetilde{f}(\omega) \text{ and } \mathcal{F}^{-1}[\widetilde{f}^{(\alpha)}(\omega)]=(-it)^\alpha f(t).
	\end{equation}
\end{Theorem}

\begin{proof}
	\begin{equation}
	\begin{split}
	\mathcal{F}[f^{(\alpha)}(t)]&=\frac{1}{\sqrt{2\pi}}\int_\mathbb{R} f^{(\alpha)}(t)e^{-i\omega t}dt=\frac{1}{\sqrt{2\pi}}\int_{\mathbb{R}}\bigg[ \int_{\mathbb{R}}f(\tau)\delta^{(\alpha)}(t-\tau)d\tau \bigg]e^{-i\omega t} dt\\
	&\frac{1}{\sqrt{2\pi}}=\int_{\mathbb{R}} f(\tau)\bigg[\int_{\mathbb{R}}\delta^{(\alpha)}(t-\tau)e^{-i\omega t}dt\bigg]d\tau =\frac{1}{\sqrt{2\pi}}\int_{\mathbb{R}}f(t)(i\omega)^\alpha e^{-i\omega t}dt\\
	&=(i\omega)^\alpha \frac{1}{\sqrt{2\pi}}\int_{\mathbb{R}}f(t)e^{-i\omega t}dt=(i\omega)^\alpha \widetilde{f}(\omega).
	\end{split}
	\end{equation}
	The exchange of integration is justified by Fubini. The reverse direction follows by the same process.
\end{proof}

These results follow with what has been surmised, and often used to develop fractional operators. It should be mentioned, however, that in the case of the fractional Laplacian (developed from an inverse Fourier transform), these results differ. Often, a fractional Laplacian may be defined as $(-\Delta)^\alpha =\mathcal{F}^{-1}|\omega|^{2\alpha}\mathcal{F}$ since $\mathcal{F}[-\Delta]=-\mathcal{F}[\Delta]=-(i\omega)^2=\omega^2=|\omega|^2$. It should be noted, however, that for fractional powers, $\omega^{2\alpha}\neq |\omega|^{2\alpha}$ (certainly for any $\omega\notin[0,\infty)$) and this is where the two definitions differ, as the power of the differintegral may be extend to all $\alpha\in\mathbb{C}$ giving $\mathcal{F}[(-\Delta)^{\alpha}]=-(i\omega)^{2\alpha}\neq |\omega|^{2\alpha}$.

\section{Fractional wave equation}

While all of the above might seem ``nice'', it is always important to ask the question, ``How can this be applied?''. Fractional differential equations have recently been increasing in popularity due to their ability to model complex systems.

We observe the wave equation
\begin{equation}
\frac{\partial^2 f}{\partial t^2}-\frac{\partial^2 f}{\partial x^2}=0
\end{equation}
which when utilizing the change of variables $u=x+t$, $v=x-t$ becomes
\begin{equation}
\frac{\partial^2 f}{\partial u\partial v}=0.
\end{equation}

Since the differintegral defines the fractional integral and derivatives of distributions, we seek a fundamental solution to the above. By observation, this fundamental solution is $\frac{1}{2}H(u)H(v)$ since $\frac{\partial^2}{\partial u\partial v}\frac{1}{2}H(u)H(v)=\frac{1}{2}\delta(u)\delta(v)=\delta(u,v)$. Shifting back with the change of variables, we receive the fundamental solution of $\frac{1}{2}H(x+t)H(x-t)$.

\subsection{Fractionalization in $u,v$ space}

Since the equation is easier to solve in the $u,v$ variables, we fractionalize that part, giving the equation
\begin{equation}
\frac{\partial^\beta}{\partial u^\beta}\frac{\partial^\alpha }{\partial v^\alpha}f=0.
\end{equation}
This may be equivalently written as
\begin{equation}
\int_{\mathbb{R}}\bigg[\int_{\mathbb{R}}f(z,y)\delta^{(\alpha)}(v-y)dy\bigg]\delta^{(\beta)}(u-z)dz=0.
\end{equation}
In this form, it is easy to find a fundamental solution, namely
\begin{equation}
\frac{1}{2}\delta^{(-\beta)}(u)\delta^{(-\alpha)}(v)=\frac{u^{\beta-1}v^{\alpha-1}}{2\Gamma(\beta)\Gamma(\alpha)}H(u)H(v)=\frac{(x+t)^{\beta-1}(x-t)^{\alpha-1}}{2\Gamma(\beta)\Gamma(\alpha)}H(x+t)H(x-t).
\end{equation}
We see as $\alpha,\beta\to 1$ we achieve the fundamental solution for the regular wave equation.

Utilizing this fundamental solution we recover a final solution of the form
\begin{equation}
f(x,t)=\frac{(x-t)^{\alpha-1}}{\Gamma(\alpha)}\phi(x+t)+\frac{(x+t)^{\beta-1}}{\Gamma(\beta)}\psi(x-t)
\end{equation}
where $\phi,\psi$ are arbitrary and fixed for specific initial/boundary conditions. Note the following cases:

\begin{enumerate}
	\item When $\alpha=\beta=1$ we recover D'Alembert's formula,
	\begin{equation}
	f(x,t)=\phi(x+t)+\psi(x-t).
	\end{equation}
	\item When $\alpha=1,\beta=0$ we recover
	\begin{equation}
	f(x,t)=\phi(x+t)
	\end{equation}
	which is equivalent to the solution of the differential equation $\frac{\partial f}{\partial t}-\frac{\partial f}{\partial x}=0$.
	\item When $\alpha=0,\beta=1$ we recover
	\begin{equation}
	f(x,t)=\psi(x-t)
	\end{equation}
	which is equivalent to the solution of the differential equation $\frac{\partial f}{\partial t}+\frac{\partial f}{\partial x}=0$.
	\item When $\alpha=\beta=0$ we recover
	\begin{equation}
	f(x,t)=0
	\end{equation}
	which is equivalent to the solution of the differential equation $f=0$. No surprise there.
\end{enumerate}

From the above we recognize that the equivalence between fractional equations of

\begin{equation}
\frac{\partial^\beta}{\partial u^\beta}\frac{\partial^\alpha}{\partial v^\alpha}f=0\iff\Big(\frac{\partial}{\partial x}+\frac{\partial}{\partial t}\Big)^\beta\Big(\frac{\partial}{\partial x}-\frac{\partial}{\partial t}\Big)^\alpha f=0
\end{equation}
and hence we must utilize the generalized binomial formula to understand the fractional differential equation of the form
\begin{equation}
\bigg[\sum_{k=0}^\infty \frac{\Gamma(\beta+1)}{\Gamma(\beta-k+1)\Gamma(k+1)}\frac{\partial^{\beta-k}}{\partial x^{\beta-k}}\frac{\partial^k}{\partial t^k}\bigg]\bigg[\sum_{k=0}^\infty \frac{(-1)^k\Gamma(\alpha+1)}{\Gamma(\alpha-k+1)\Gamma(k+1)}\frac{\partial^{\alpha-k}}{\partial x^{-k}}\frac{\partial^k}{\partial t^k}\bigg]f=0.
\end{equation}
Observe that since these sums contain infinite (positive and negative) powers of partial derivatives, solutions to fractional differential equations are forced to be infinitely differentiable. Of course, this is always possible when searching for a distributional solution, but sometimes a classical solution may be very difficult to find.

Let us observe one solution to the fractional wave equation with prescribed initial conditions. 

\begin{Theorem}
	For $\alpha\in(0,1]$ and $\beta\in(1-\alpha,1]$ for uniqueness, the solution to the fractional wave equation
	\begin{equation}
	\Big(\frac{\partial}{\partial t}+\frac{\partial}{\partial x}\Big)^\beta\Big(\frac{\partial}{\partial t}-\frac{\partial}{\partial x}\Big)^\alpha f=0
	\end{equation}
	with initial conditions $f(x,0)=g(x)$ and $f_t(x,0)=h(x)$ is
	\begin{equation}
	\begin{split}
	f(x,t)&=\frac{1}{2}\Big(\frac{x+t}{x-t}\Big)^{1-\alpha}\bigg[g(x+t)+\int_0^{x+t}h(y)+y^{\alpha+\beta-3}\bigg(\int_0^y (\alpha-\beta)z^{2-\alpha-\beta}g'(z)+(\alpha+\beta-2)z^{2-\alpha-\beta}h(z)dz\bigg)dy\bigg]\\
	&+\frac{1}{2}\Big(\frac{x-t}{x+t}\Big)^{1-\beta}\bigg[g(x-t)-\int_0^{x-t}h(y)+y^{\alpha+\beta-3}\bigg(\int_0^y (\alpha-\beta)z^{2-\alpha-\beta}g'(z)+(\alpha+\beta-2)z^{2-\alpha-\beta}h(z)dz\bigg)dy\bigg]
	\end{split}
	\end{equation}
	where the integrals $\int_0^x (\circ)dy$ are considered as inverse derivatives.
\end{Theorem}
\begin{proof}
Note
\begin{equation}
f(x,0)=(-x)^{\alpha-1}\phi(x)+x^{\beta-1}\psi(-x)=g(x)
\end{equation}
\begin{equation}
f_t(x,0)=(\alpha-1)(-x)^{\alpha-2}\phi(x)+(-x)^{\alpha-1}\phi'(x)+(\beta-1)x^{\beta-2}\psi(-x)+x^{\beta-1}\psi'(-x)=h(x)
\end{equation}

Noting that for $\alpha=\beta=1$ the solution comes in the form
\begin{equation}
\phi(x+t)=\frac{1}{2}\bigg[g(x+t)+\int_0^{x+t}h(y)dy\bigg] \text{ and }\psi(x-t)=\frac{1}{2}\bigg[g(x-t)-\int_0^{x-t}h(y)dy\bigg]
\end{equation}
so we surmise that for arbitrary $\alpha,\beta$ we recover a solution in the form
\begin{equation}
f(x,t)=\frac{1}{2}\Big(\frac{x+t}{x-t}\Big)^{1-\alpha}\bigg[g(x+t)+\int_0^{x+t}h(y)+\eta(y)dy\bigg]+\frac{1}{2}\Big(\frac{x-t}{x+t}\Big)^{1-\beta}\bigg[g(x-t)-\int_0^{x-t}h(y)+\nu(y)dy\bigg]
\end{equation}
for some functions $\eta,\nu$ dependent on $\alpha,\beta$. We utilize the first identity, giving
\begin{equation}
f(x,0)=g(x)+\int_0^x h(y)+\eta(y)dy-\int_0^x h(y)+\nu(y)dy=g(x)+\int_0^x\eta(y)-\nu(y)dy=g(x).
\end{equation}
Hence it is clear that $\eta=\nu$. The second identity leads to
\begin{equation}
f_t(x,0)=h(x)+\eta(x)+\frac{(1-\alpha)}{x}g(x)+\frac{(1-\alpha)}{x}\int_0^x h(y)+\eta(y)dy+\frac{(\beta-1)}{x}g(x)+\frac{(1-\beta)}{x}\int_0^x h(y)+\eta(y)dy=h(x)
\end{equation}
or rather
\begin{equation}
x\eta(x)+(2-\alpha-\beta)\int_0^x \eta(y)dy=(\alpha-\beta)g(x)+(\alpha+\beta-2)\int_0^x h(y)dy
\end{equation}
which becomes the differential equation
\begin{equation}
(3-\alpha-\beta)\eta(x)+x\eta'(x)=(\alpha-\beta)g'(x)+(\alpha+\beta-2)h(x).
\end{equation}
Elementary methods (as this is a first order linear differential equation) give the solution of this equation to be
\begin{equation}
\eta(x)=x^{\alpha+\beta-3}\int_0^x (\alpha-\beta)y^{2-\alpha-\beta}g'(y)+(\alpha+\beta-2)y^{2-\alpha-\beta}h(y)dy.
\end{equation}
Note that we consider $\int_0^x (\circ) dy$ as the \textit{inverse derivative} rather than an integral. This side-steps requiring that the integral converges at the lower bound. See \cite{Camrud} for the definition of inverse derivative.
\end{proof}

We now compute this result for $f(x,0)=\sin(x)$ and $f_t(x,0)=\cos(x)$. Then
\begin{equation}
\begin{split}
f(x,t)&=\frac{1}{2}\Big(\frac{x+t}{x-t}\Big)^{1-\alpha}\bigg[\sin(x+t)+\int_0^{x+t}\cos(y)+2(\alpha-1)y^{\alpha+\beta-3}\bigg(\int_0^y z^{2-\alpha-\beta}\cos(z)dz\bigg)dy\bigg]\\
&+\frac{1}{2}\Big(\frac{x-t}{x+t}\Big)^{1-\beta}\bigg[\sin(x-t)-\int_0^{x-t}\cos(y)+2(\alpha-1)y^{\alpha+\beta-3}\bigg(\int_0^y z^{2-\alpha-\beta}\cos(z)dz\bigg)dy\bigg]
\end{split}
\end{equation}

Figures 1,2,3,4 show a three-dimensional graph of $Re[f]$ versus time and space for chosen values of $\alpha=\beta=0,\frac{1}{2},\frac{3}{4},1$ respectively.

\subsection{Fractionalized in $x,t$ space}

Instead of fractionalizing the wave equation as
\begin{equation}
\frac{\partial^\beta}{\partial u^\beta}\frac{\partial^\alpha}{\partial v^\alpha}f=0
\end{equation}
we can fractionalize it as
\begin{equation}
\frac{\partial^\beta}{\partial t^\beta}f-\frac{\partial^\alpha}{\partial x^\alpha}f=0.
\end{equation}

\begin{Theorem}
	A general solution to the fractional wave equation
	\begin{equation}
	\frac{\partial^\beta}{\partial t^\beta}f-\frac{\partial^\alpha}{\partial x^\alpha}f=0.
	\end{equation}
	is
	\begin{equation}
	f(x,t)=\sum_{k=0}^\infty c_k e^{\omega^\alpha t+\omega^\beta x}.
	\end{equation}
\end{Theorem}

\begin{proof}

Since $e^{\omega y}$ is an eigenfunction of $\text{\Large\S}_y^\alpha$ with eigenvalue $\omega^{-\alpha}$, we assume that our solution is time-periodic. That is
\begin{equation}
f(x,t)=e^{\omega t}g(x)
\end{equation}
and we recover the eigenvalue equation
\begin{equation}
\frac{\partial^\beta}{\partial t^\beta}f-\frac{\partial^\alpha}{\partial x^\alpha}f=0\iff \frac{\partial^\alpha}{\partial x^\alpha}f(x,t)=\omega^{\beta}f(x,t).
\end{equation}
Of course this may be manipulated further into
\begin{equation}
\frac{d^\alpha}{dx^\alpha}g(x)=\omega^{\beta}g(x)
\end{equation}
which has the solution
\begin{equation}
g(x)=c\cdot e^{\omega^{\beta/\alpha}x}
\end{equation}
for $c\in\mathbb{C}$. Therefore the solution to the differential equation is of the form
\begin{equation}
f(x,t)=\sum_{k=0}^\infty c_k e^{\omega^\alpha t+\omega^\beta x}
\end{equation}
where we took $\omega\mapsto\omega^\alpha$ to simplify the expression.

\end{proof}

\begin{Remark}
	For any $\alpha,\beta\notin\mathbb{Z}$ the expressions $\omega^\alpha,\omega^\beta$ are multivalued. Even further, for $\alpha,\beta\notin\mathbb{Q}$ the expressions $\omega^\alpha,\omega^\beta$ have infinitely many values. For this reason, it is important to study the possibilities for the function $f(x,t)=\sum_{k=0}^\infty c_k e^{\omega^\alpha t+\omega^\beta x}$, as this is not a simple Fourier series.
	
	Further, even for $\alpha,\beta\in\mathbb{Z}$, alternate solutions are of the form
	\begin{equation}
	f(x,t)=\sum_{k=0}^\infty c_k e^{\omega t+\omega^{\beta/\alpha} x} \text{ and }
	f(x,t)=\sum_{k=0}^\infty c_k e^{\omega^{\alpha/\beta} t+\omega x}
	\end{equation}
	where in both cases $\omega^{\beta/\alpha}=(\omega^\beta)^{1/\alpha},\omega^{\alpha/\beta}=(\omega^\alpha)^{1/\beta}$ are multivalued expressions. This is precisely the case for $\alpha=\beta=2$ (the regular wave equation) when we recover solutions in the form
	\begin{equation}
	f(x,t)=\sum_{k=0}^\infty c_k e^{i\omega (x+t)}+\sum_{k=0}^\infty c_k e^{i\omega (x-t)}=F(x+t)+G(x-t)
	\end{equation}
	D'Alembert's formula.
\end{Remark}

We compute this result for $f(x,0)=\sin(x)$. This, along with assuming periodicity, provides a well-defined solution:

\begin{equation}
f(x,t)=\frac{1}{2i}\Big[e^{i^{\alpha/\beta} t+i x}-e^{(-i)^{\alpha/\beta} t-i x}\Big]
\end{equation}

Figures 5,6,7 show a three-dimensional graph of $f$ versus time and space for chosen values of $\alpha/\beta<1,\alpha/\beta>1,\alpha/\beta=1$ respectively.

\newpage

\begin{figure}
	\includegraphics{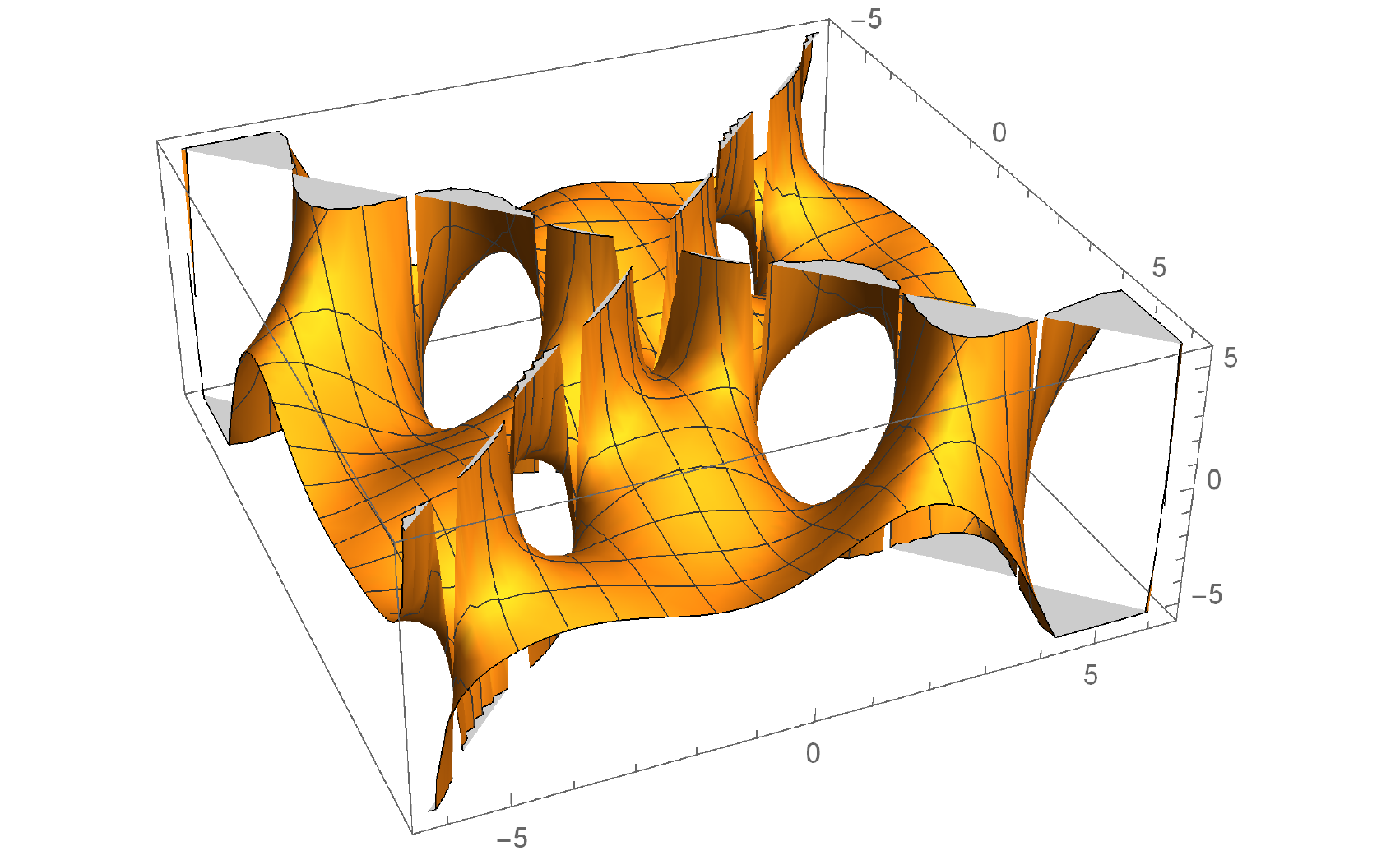}
	\caption[Figure 1.]{Solution for $\alpha=\beta=0$. This is an overdetermined system, so it is hard to argue what the results imply.}
\end{figure}

\begin{figure}
	\includegraphics{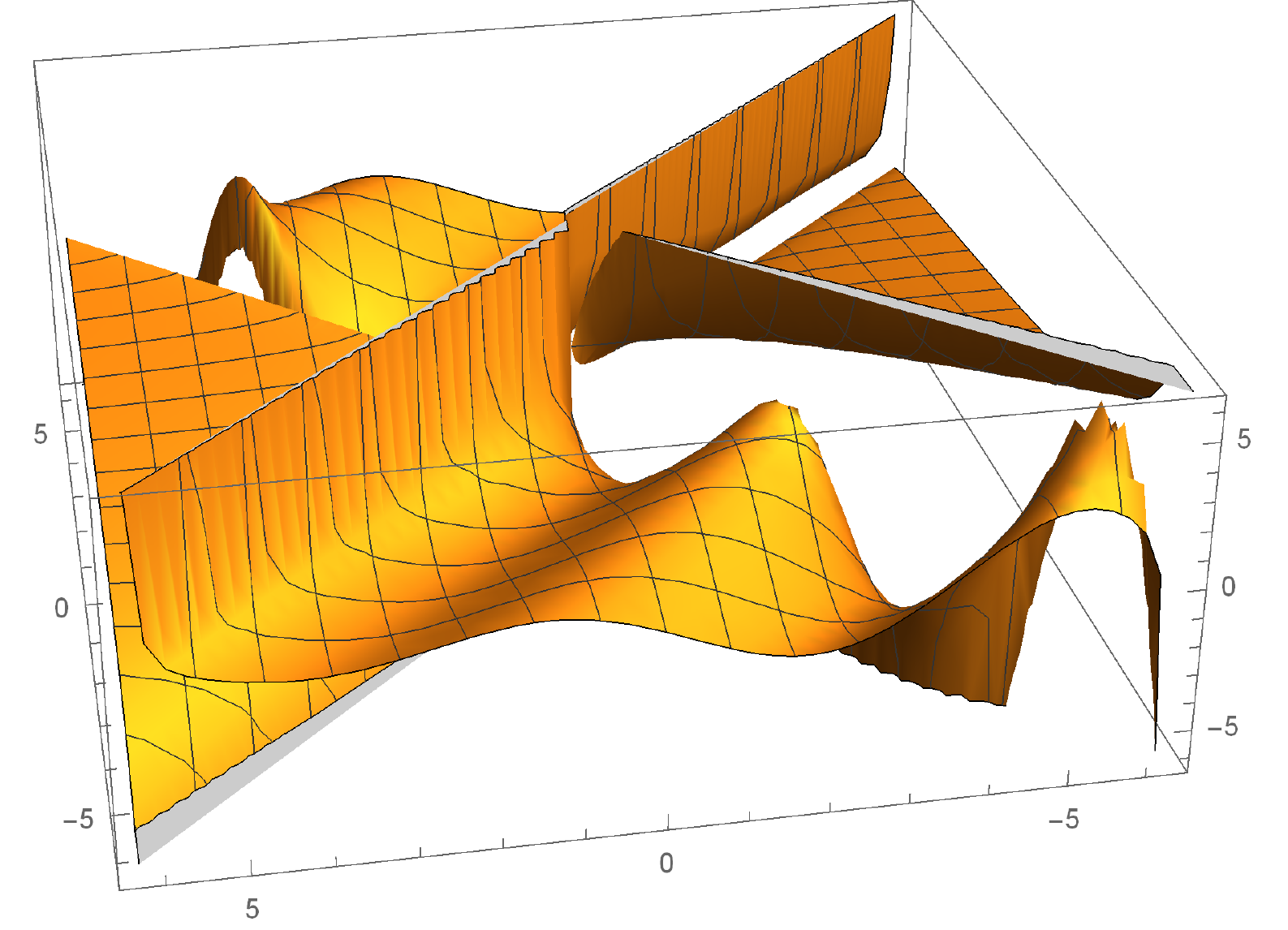}
	\caption[Figure 2.]{Solution for $\alpha=\beta=\frac{1}{2}$. One can see the light cone with minimal real behavior for $x<t$.}
\end{figure}

\begin{figure}
	\includegraphics{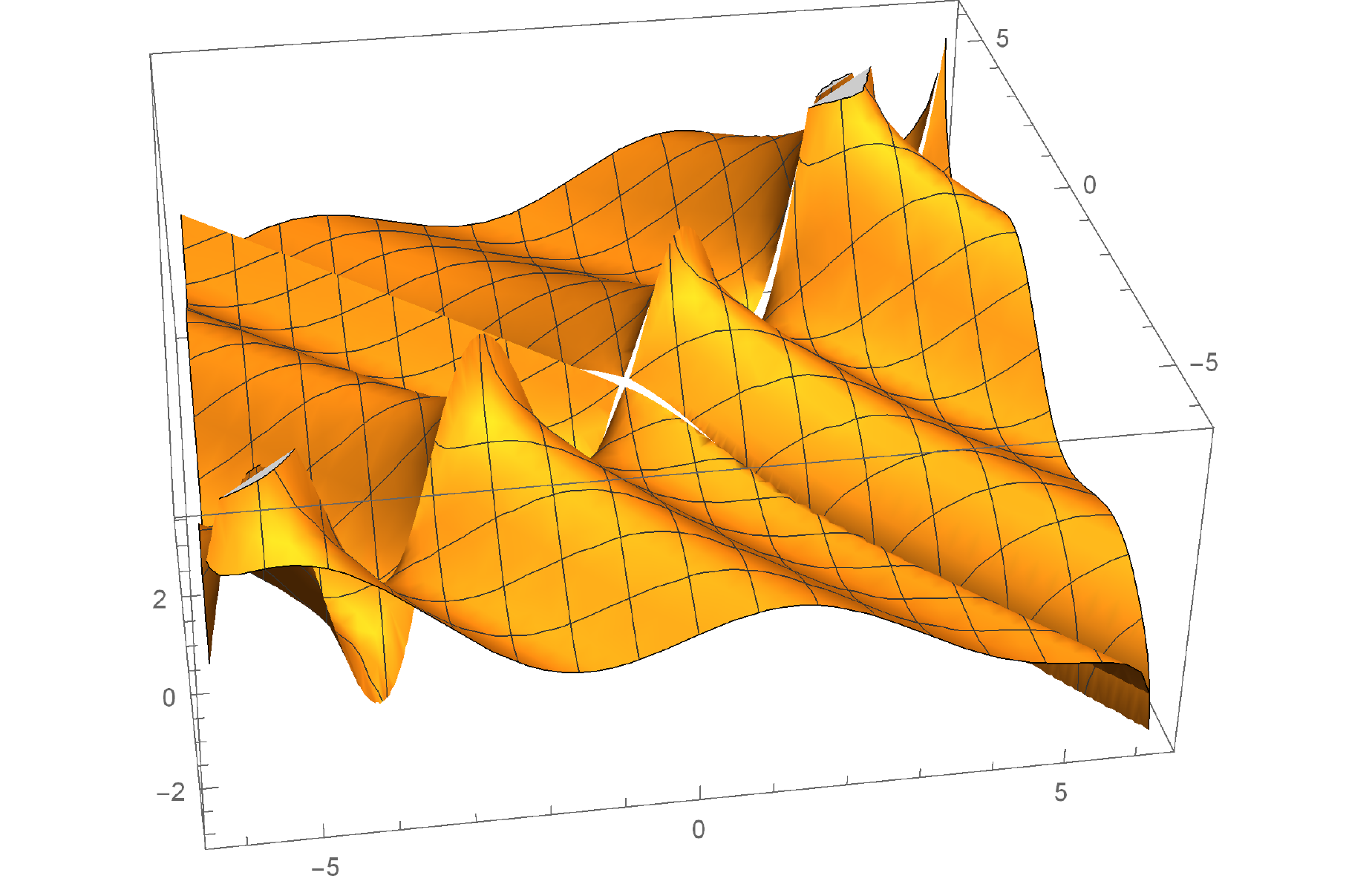}
	\caption[Figure 3.]{Solution for $\alpha=\beta=\frac{3}{4}$. The light cone appears again, but with much more behavior for $x<t$.}
\end{figure}

\begin{figure}
	\includegraphics{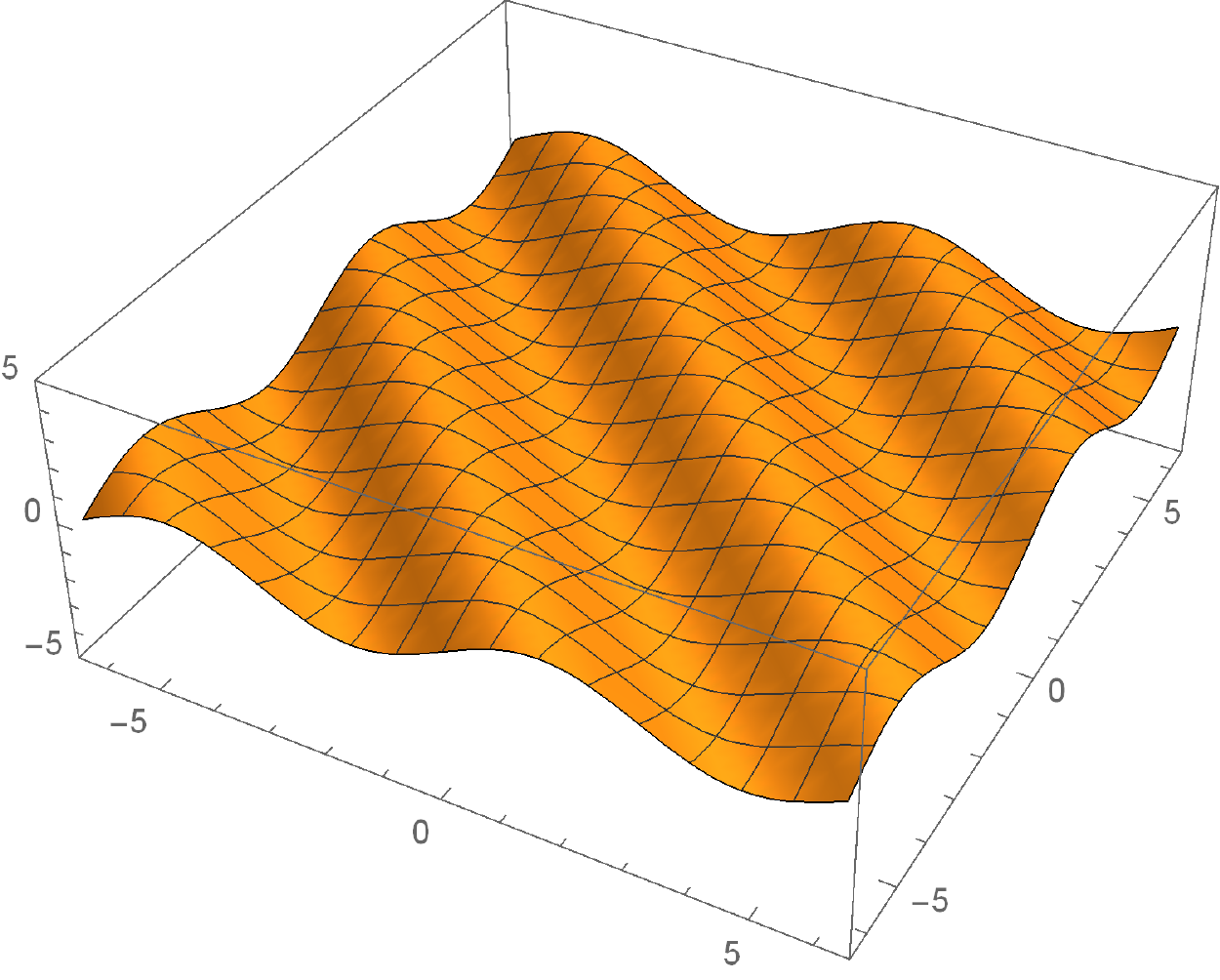}
	\caption[Figure 4.]{Solution for $\alpha=\beta=1$. The light cone has disappeared, and we recover the normal wave equation solution}
\end{figure}

\begin{figure}
	\includegraphics{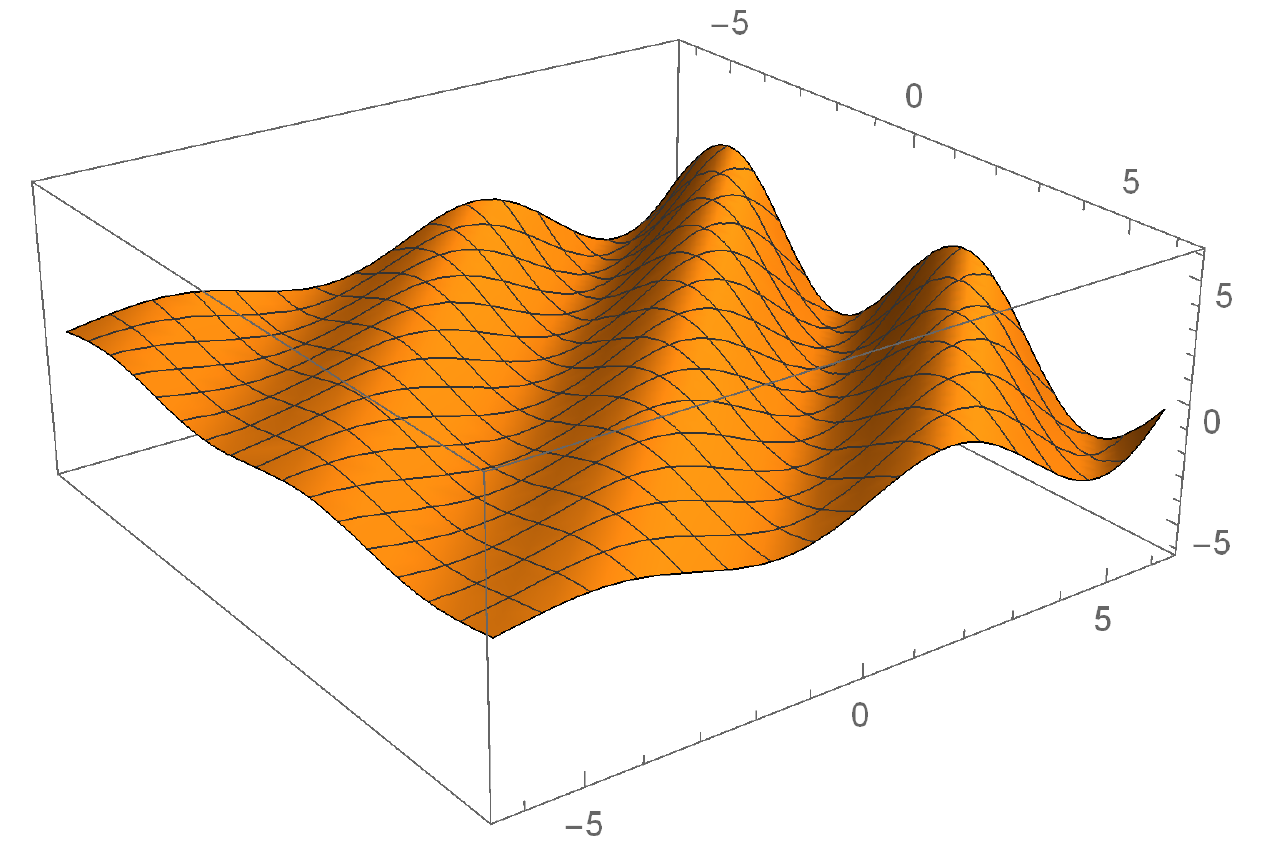}
	\caption[Figure 5.]{Solution for $\alpha/\beta<1$. There seems to be increasing amplitude over time. So-called ``negative damping''.}
\end{figure}

\begin{figure}
	\includegraphics{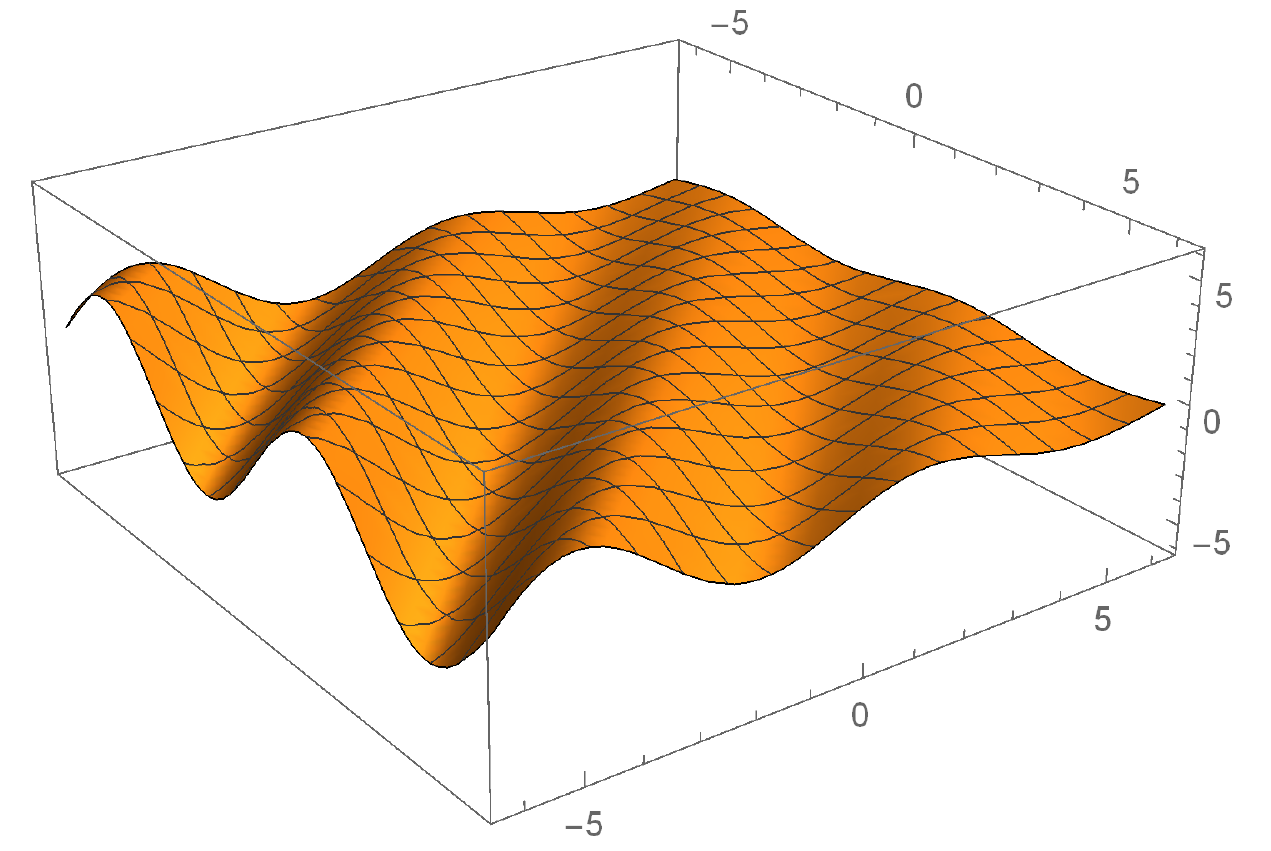}
	\caption[Figure 6.]{Solution for $\alpha/\beta>1$. We see precisely the opposite of the prior case, where there is positive damping over time.}
\end{figure}

\begin{figure}
	\includegraphics{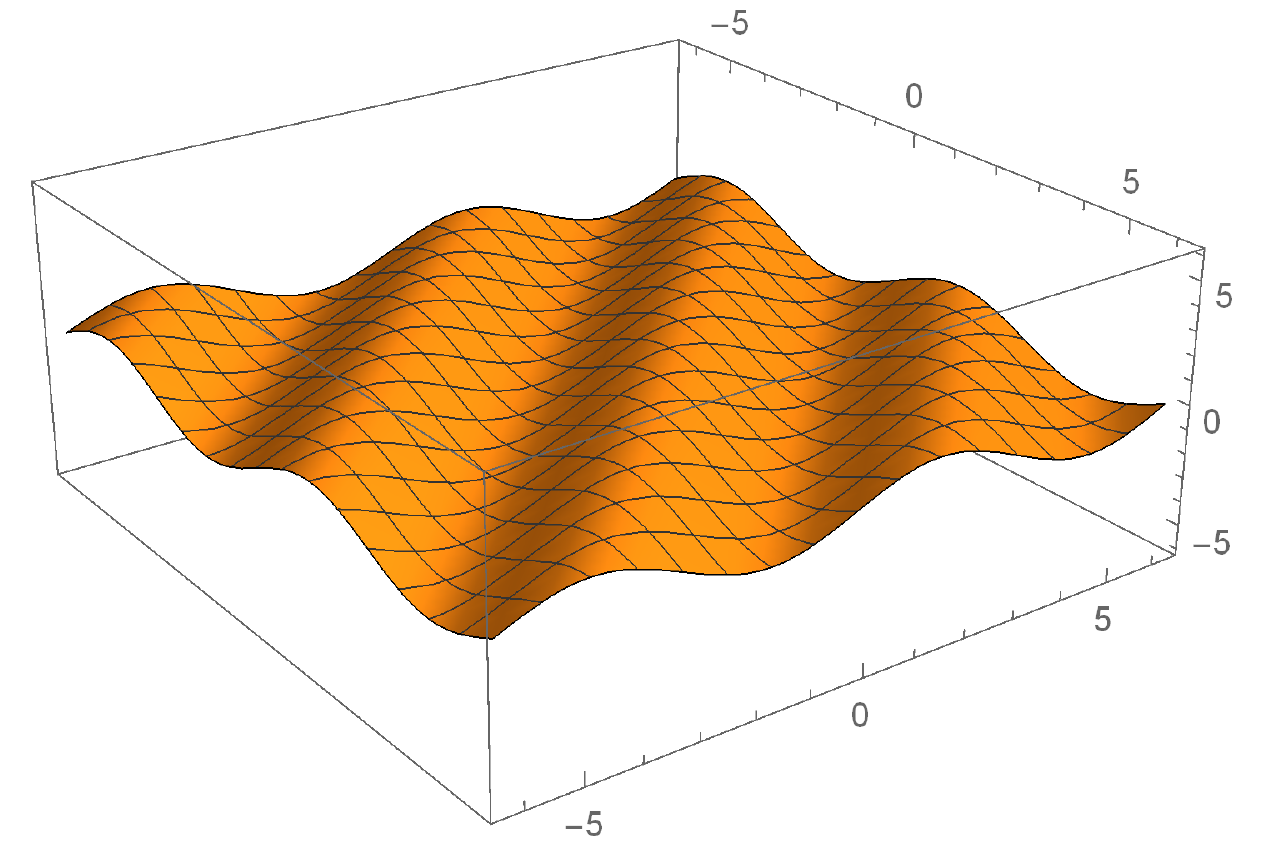}
	\caption[Figure 7.]{Solution for $\alpha/\beta=1$. There is no damping, such as the regular wave equation solution would imply .}
\end{figure}


\begin{thebibliography}{References}
	\bibitem{Camrud} Camrud, E. A novel approach to fractional calculus: utilizing fractional integrals and derivatives of the Dirac delta function. \textit{Progress in fractional differentiation and applications.} Accepted 2 Mar. 2018.
\end{thebibliography}
\end{document}